\documentclass [twoside,reqno,12pt] {amsart}

\usepackage{amsfonts}
\usepackage{amssymb}
\usepackage{a4}

\newtheorem{thm}{Theorem}[section]

\newtheorem{prop}[thm]{Proposition}

\theoremstyle{definition}

\theoremstyle{remark}
\newtheorem{rem}{Remark}[section]

\numberwithin{equation}{section}

\renewcommand{\Re}{\hbox{Re}\,}
\renewcommand{\Im}{\hbox{Im}\,}

\newcommand{\C}{\mathbb{C}}

\renewcommand{\div}{\operatorname{div}}

\newcommand{\R}{\mathbb{R}}

\def\hat{\widehat}
\def\tilde{\widetilde}
\def \bfo {\begin {eqnarray*} }
\def \efo {\end {eqnarray*} }
\def \ba {\begin {eqnarray*} }
\def \ea {\end {eqnarray*} }
\def \beq {\begin {eqnarray}}
\def \eeq {\end {eqnarray}}

\def \p {\partial}

\def\hat{\widehat}
\def\tilde{\widetilde}
\def \bfo {\begin {eqnarray*} }
\def \efo {\end {eqnarray*} }
\def \ba {\begin {eqnarray*} }
\def \ea {\end {eqnarray*} }
\def \beq {\begin {eqnarray}}
\def \eeq {\end {eqnarray}}

\def \p {\partial}

\begin{document}

 \title[Inverse problems for polyharmonic operator ]{Inverse boundary value problems for the perturbed polyharmonic operator}

\author[Krupchyk]{Katsiaryna Krupchyk}

\address
        {K. Krupchyk, Department of Mathematics and Statistics \\
         University of Helsinki\\
         P.O. Box 68 \\
         FI-00014   Helsinki\\
         Finland}

\email{katya.krupchyk@helsinki.fi}

\author[Lassas]{Matti Lassas}

\address
        {M. Lassas, Department of Mathematics and Statistics \\
         University of Helsinki\\
         P.O. Box 68 \\
         FI-00014   Helsinki\\
         Finland}

\email{matti.lassas@helsinki.fi}

\author[Uhlmann]{Gunther Uhlmann}

\address
       {G. Uhlmann, Department of Mathematics\\
       University of Washington\\
       Seattle, WA  98195-4350\\
       and
       Department of Mathematics\\ 
       340 Rowland Hall\\
        University of California\\
        Irvine, CA 92697-3875\\
       USA}
\email{gunther@math.washington.edu}

\maketitle

\begin{abstract} 

We show that a first order perturbation $A(x)\cdot D+q(x)$Ê of the polyharmonic operator $(-\Delta)^m$, $m\ge 2$, can be determined uniquely from the set of the Cauchy data for the perturbed polyharmonic operator on a bounded domain in $\R^n$, $n\ge 3$.  Notice that the corresponding result does not hold in general when $m=1$.

\end{abstract}

\section{Introduction and statement of results}

Let $\Omega\subset \R^n$, $n\ge 3$,  be a bounded domain with $C^\infty$-boundary, and consider the polyharmonic operator $(-\Delta)^m$, $m\ge 1$, which is 
a positive self-adjoint operator on $L^2(\Omega)$ with the domain 
\[
H^{2m}(\Omega)\cap H^{m}_0(\Omega),\quad H^m_0(\Omega)=\{u\in H^{m}(\Omega):\gamma u=0\}.
\]
Here 
\[
\gamma u=(u|_{\p\Omega},\p_{\nu}u|_{\p \Omega},\dots,\p_{\nu}^{m-1}u|_{\p \Omega})
\]
is the Dirichlet trace of $u$, $\nu$ is the unit outer normal to the boundary $\p \Omega$, and $H^s(\Omega)$ is the standard Sobolev space on $\Omega$, $s\in\R$.

Consider the operator 
\[
\mathcal{L}_{A,q}(x,D)=(-\Delta)^m+\sum_{j=1}^n A_j(x)D_j+q(x)=(-\Delta)^m+A(x)\cdot D+q(x),
\]
with $A=(A_j)_{1\le j\le n}\in W^{1,\infty}(\Omega,\C^n)$ and $q\in L^\infty(\Omega,\C)$. Here $D=i^{-1}\nabla$. 
Viewed as an unbounded operator on $L^2(\Omega)$ and equipped with the domain $H^{2m}(\Omega)\cap H^{m}_0(\Omega)$, 
the operator $\mathcal{L}_{A,q}$  is closed with purely discrete spectrum. 

Let us make the following assumption,
\begin{itemize}
\item[\textbf{(A)}] $0$ is not an eigenvalue of $\mathcal{L}_{A,q}(x,D):H^{2m}(\Omega)\cap H^{m}_0(\Omega)\to L^2(\Omega)$. 
\end{itemize}
Under the assumption (A), for any $f=(f_0,f_1,\dots,f_{m-1})\in 
\mathcal{H}^{0,m-1}(\p \Omega):= \prod_{j=0}^{m-1}H^{2m-j-1/2}(\p \Omega)$,  
the Dirichlet problem 
\begin{equation}
\label{eq_Dirichlet_problem}
\begin{aligned}
\mathcal{L}_{A,q} u&=0\quad \textrm{in} \quad \Omega, \\
\gamma u&=f\quad \textrm{on} \quad \p \Omega,
\end{aligned}
\end{equation}
has a unique solution $u\in H^{2m}(\Omega)$. Introducing the Neumann trace operator $\tilde\gamma$ by 
\begin{align*}
&\tilde \gamma: H^{2m}(\Omega)\to \mathcal{H}^{m,2m-1}(\p \Omega):= \prod_{j=m}^{2m-1}H^{2m-j-1/2}(\p \Omega),\\
&\tilde \gamma u=(\p_{\nu}^{m}u|_{\p \Omega},\dots,\p_{\nu}^{2m-1}u|_{\p \Omega}),
\end{align*} 
see \cite{Grubbbook2009}, 
 we define the Dirichlet--to--Neumann map $\mathcal{N}_{A,q}=\mathcal{N}_{A,q}^{\Omega}$ by
\[
\mathcal{N}_{A,q}: \mathcal{H}^{0,m-1}(\p \Omega)\to \mathcal{H}^{m,2m-1}(\p \Omega), \quad \mathcal{N}_{A,q}(f)=\tilde \gamma u,
\]
where $u\in H^{2m}(\Omega)$ is the solution to the Dirichlet problem  \eqref{eq_Dirichlet_problem}.  For future references, 
let us also introduce the set of the Cauchy data $\mathcal{C}_{A,q}=\mathcal{C}_{A,q}^{\Omega}$ for the operator $\mathcal{L}_{A,q}$ defined as follows,
\[
\mathcal{C}_{A,q}=\{(\gamma u,\tilde \gamma u): u\in H^{2m}(\Omega), \mathcal{L}_{A,q}u=0\ \textrm{in }\Omega\}. 
\]
When the assumption (A) holds,  the set $\mathcal{C}_{A,q}$ is the graph of the Dirichlet--to--Neumann map $\mathcal{N}_{A,q}$.

In this paper we are concerned with  the inverse problem of recovering  the first order perturbation $A(x)\cdot D+ q(x)$ in $\Omega$  from the knowledge of the Dirichlet--to--Neumann map $\mathcal{N}_{A,q}$ on the boundary of $\Omega$. 

As it  was noticed in \cite{Sun_1993}, when $m=1$, we have 
\[
e^{-i\psi}\mathcal{L}_{A,q}e^{i\psi}=\mathcal{L}_{\tilde A, \tilde q},\quad e^{i\psi}\mathcal{N}_{ \tilde A, \tilde q} e^{-i\psi}=\mathcal{N}_{A,q} -i\p_{\nu}\psi,
\]
\begin{equation}
\label{eq_A_introduction}
  \tilde A=A+2\nabla\psi,\quad  \tilde q=q+A\cdot \nabla \psi + |\nabla \psi|^2-i\Delta\psi. 
\end{equation}
Thus, $\mathcal{N}_{A, q} =\mathcal{N}_{\tilde A,\tilde q}$ when $\psi\in C^2(\overline \Omega)$ is such that $\psi|_{\p\Omega}=\p_{\nu}\psi|_{\p \Omega}=0$. We may therefore  hope to recover the coefficients 
$A$ and $q$ from boundary measurements only modulo a gauge transformation in \eqref{eq_A_introduction}.

Starting with the pioneering paper \cite{Sun_1993}, inverse boundary value problems for first order perturbations of the Laplacian have been extensively studied, usually in the context of magnetic Schr\"odinger operators. 
In \cite{Sun_1993} it was shown that  the hope mentioned above is justified, and the magnetic field and the electric potential are uniquely determined by the Dirichlet--to-Neumann map, provided that the magnetic field is small
in a suitable sense. The smallness condition was removed in \cite{NakSunUlm_1995} in the case of $C^\infty$-smooth magnetic and electric potentials, and also for $C^2$-compactly supported magnetic and $L^\infty$ 
electric potentials, see also \cite{ChNS_2001}.  The regularity assumptions on the potentials were subsequently weakened in the works \cite{Salo_diss, Tolmasky_1998}. 
 The case of partial boundary measurements for the 
magnetic Schr\"odinger operator was studied in \cite{DKSU_2007, KnuSalo_2007}.

The purpose of this paper is to show that the obstruction to uniqueness coming from the gauge invariance \eqref{eq_A_introduction} when $m=1$,   can be eliminated by passing to operators of higher order, and the unique determination of the coefficients $A$ and $q$ becomes then possible. Throughout the paper we shall assume therefore that  $m\ge 2$.  Our first result is as follows. 

\begin{thm}
\label{thm_main} Let $\Omega\subset \R^n$, $n\ge 3$,  be a bounded domain with  $C^\infty$-boundary, and let $A^{(1)}, A^{(2)}\in W^{1,\infty}(\R^n,\C^n)\cap \mathcal{E}'(\overline\Omega,\C^n)$ and $q^{(1)},q^{(2)}\in L^\infty(\Omega,\C)$, be such that the assumption (A) is satisfied for both operators. 
If $\mathcal{N}_{A^{(1)},q^{(1)}}=\mathcal{N}_{A^{(2)},q^{(2)}}$, then $A^{(1)}=A^{(2)}$ and $q^{(1)}=q^{(2)}$ in $\Omega$.  
\end{thm}

The proof of Theorem \ref{thm_main} exploits complex geometric optics solutions for the equations $\mathcal{L}_{A^{(j)},q^{(j)}}u=0$, $j=1,2$.  Following \cite{DKSU_2007, KenSjUhl2007}, 
the construction of such solutions will be carried out using Carleman estimates.  We remark here that starting with the fundamental paper \cite{Syl_Uhl_1987}, complex geometric optics solutions have been a central tool in establishing uniqueness results in elliptic inverse boundary value problems.

Dropping  the assumption that $A^{(j)}=0$, $j=1,2$,  along $\p \Omega$ in Theorem \ref{thm_main}, we have the following result. 
\begin{thm}
\label{thm_main_2} 
 Let $\Omega\subset \R^n$, $n\ge 3$,  be a bounded domain with  $C^\infty$-boundary, and let $A^{(1)}, A^{(2)}\in C^{\infty}(\overline{\Omega},\C^n)$ and $q^{(1)},q^{(2)}\in C^\infty(\overline \Omega,\C)$, be such that the assumption (A) is satisfied  for both operators.
If $\mathcal{N}_{A^{(1)},q^{(1)}}=\mathcal{N}_{A^{(2)},q^{(2)}}$, then $A^{(1)}=A^{(2)}$ and $q^{(1)}=q^{(2)}$ in $\Omega$.  
\end{thm}

The idea of the proof of Theorem \ref{thm_main_2}  is to reduce it to Theorem \ref{thm_main}  by using the boundary reconstruction of the vector field part of the perturbation.  When doing the boundary reconstruction, similarly to  \cite{KKLM, LeeUhl89, NakSunUlm_1995},  we use pseudodifferential techniques, which motivates the need to require $C^\infty$ smoothness assumptions on the coefficients of the perturbed operator.

Finally, it may be interesting to notice that the boundary reconstruction becomes unnecessary when the boundary of the domain $\Omega$ is connected. We have the following result. 

\begin{thm}
\label{thm_main_3} Let $\Omega\subset \R^n$, $n\ge 3$,  be a bounded domain with  connected $C^\infty$--boundary, and let $A^{(1)}, A^{(2)}\in W^{1,\infty}(\Omega,\C^n)$ and $q^{(1)},q^{(2)}\in L^\infty(\Omega,\C)$, be such that the assumption (A) is satisfied  for both operators.
If $\mathcal{N}_{A^{(1)},q^{(1)}}=\mathcal{N}_{A^{(2)},q^{(2)}}$, then $A^{(1)}=A^{(2)}$ and $q^{(1)}=q^{(2)}$ in $\Omega$.  
\end{thm}

The inverse boundary value problem of the recovery of a zeroth order perturbation of the biharmonic operator has been studied in \cite{Ikehata_1991, Isakov_1991}, and unique identifiability results were obtained, similarly to the case of the Schr\"odinger operator. The areas of physics and geometry where higher order operators occur, include the study of the Kirchhoff plate equation in the theory of elasticity, and the study of the Paneitz-Branson operator in 
conformal geometry, see \cite{GGS_book}. Inverse boundary value problems for differential perturbations of the biharmonic, or more generally, polyharmonic, operator are therefore natural to consider, and the present paper is meant as a step in this direction. 

We would also like to mention the important results of \cite{Eskin_2001}, concerned with the inverse boundary value problem for a first order matrix perturbation of the Laplacian, considered in a smoothly bounded convex domain
in $\R^n$, $n\ge 3$. The results of \cite{Eskin_2001} show that it is possible to recover the matrix potentials up to a gauge transformation, given by a smooth invertible matrix. In our study of the perturbed polyharmonic operator, rather than reducing it to a system, we adopt a direct approach, which has the merit that the issue of the gauge equivalence does not arise.

Finally, we would like to point out that the results obtained in this paper can be viewed as generalizations of the corresponding results  for second order equations, encountered
in electrical impedance tomography, see \cite{AP,ALP,N2, S}  for the two dimensional case, and \cite{Cal, GLU1, N1, PPU, Syl_Uhl_1987} for the higher dimensions, 
as well as in inverse boundary value problems and inverse scattering problems for the Schr\"odinger equation \cite{Bukh_2008, GLU1, Ima_Uhl_Yam_2010,  N2, No, Syl_Uhl_1987}, and in elliptic inverse problems on Riemannian manifolds,
\cite{GT1, GT2, LTU, LU, LeeUhl89}. There are also counterexamples for uniqueness of inverse problems with very non-regular electric fields \cite{GLU2} and magnetic fields \cite{GKLU2}. These 
counterexamples are closely related to the so-called invisibility cloaking, see \cite{GLU1, GKLU1, GKLUbull, KSVW, Le, PSS1}.

The plan of the paper is as follows. In Section 2 we construct complex geometric optics solutions of the equation  $\mathcal{L}_{A,q}u=0$ in $\Omega$, which are instrumental in proving Theorems \ref{thm_main}--\ref{thm_main_3}. 
The proof of Theorem \ref{thm_main} is then given in Section 3, while the boundary reconstruction of the vector field part of the perturbation and the proof of Theorem \ref{thm_main_2} are the subjects of Section 4. The proof of Theorem \ref{thm_main_3} is finally given in 
Section 5. Appendix A contains a characterization of curl-free vector fields which may be of some independent interest.

\section{Construction of complex geometric optics solutions} 

Let $\Omega\subset\R^n$, $n\ge 3$, be a bounded domain with $C^\infty$-boundary.  Following \cite{DKSU_2007, KenSjUhl2007}, we shall use the method of Carleman estimates to construct complex geometric optics solutions for the equation $\mathcal{L}_{A,q}u=0$ in $\Omega$,  with $A\in W^{1,\infty}(\Omega, \C^n)$ and $q\in L^\infty(\Omega,\C)$. 

First we shall derive a Carleman estimate for the semiclassical polyharmonic operator $(-h^2\Delta)^m$, where $h>0$ is a small parameter,  by iterating the corresponding Carleman estimate for the semiclassical Laplacian $-h^2\Delta$, which we now proceed to recall following  \cite{KenSjUhl2007}. 
Let $\tilde \Omega$ be an open set in $\R^n$ such that $ \Omega\subset\subset\tilde \Omega$ and 
$\varphi\in C^\infty(\tilde \Omega,\R)$.  Consider the conjugated operator 
\[
P_\varphi=e^{\frac{\varphi}{h}}(-h^2\Delta) e^{-\frac{\varphi}{h}}
\]
and its semiclassical principal symbol
\[
p_\varphi(x,\xi)=\xi^2+2i\nabla \varphi\cdot \xi-|\nabla \varphi|^2, \quad x\in \overline{\Omega},\quad  \xi\in \R^n. 
\]
Following \cite{KenSjUhl2007}, we say that $\varphi$ is a limiting Carleman weight for $-h^2\Delta$ in $\tilde \Omega$, if $\nabla \varphi\ne 0$ in $\tilde \Omega$ and the Poisson bracket of $\Re p_\varphi$ and $\Im p_\varphi$ satisfies, 
\[
\{\Re p_\varphi,\Im p_\varphi\}(x,\xi)=0 \quad \textrm{when}\quad p_\varphi(x,\xi)=0, \quad (x,\xi)\in \overline{\Omega}\times \R^n. 
\]
Examples are  linear weights $\varphi(x)=\alpha\cdot x$, $\alpha\in \R^n$, $|\alpha|=1$, and logarithmic weights $\varphi(x)=\log|x-x_0|$, with $x_0\not\in \tilde \Omega$.  In this paper we shall only be concerned with the case of linear weights.

In what follows we shall equip  the standard Sobolev space $H^s(\R^n)$, $s\in\R$, with the semiclassical norm $\|u\|_{H^s_{\textrm{scl}}}=\|\langle hD \rangle^s u\|_{L^2}$. Here $\langle\xi  \rangle=(1+|\xi|^2)^{1/2}$. 
We shall need the following result, obtained in \cite{KenSjUhl2007}.

\begin{prop}
Let $\varphi$ be a limiting Carleman weight for the semiclassical Laplacian on $\tilde \Omega$. Then the Carleman estimate 
\begin{equation}
\label{eq_Carleman_lap}
\|e^{\frac{\varphi}{h}}(-h^2\Delta)e^{-\frac{\varphi}{h}}u\|_{H^s_{\emph{scl}}}\ge \frac{h}{C_{s,\Omega}}\|u\|_{H^{s+1}_{\emph{scl}}},
\end{equation}
holds for all $u\in C^\infty_0(\Omega)$, $s\in\R$ and  all $h>0$ small enough.  
\end{prop}

Iterating the Carleman estimate \eqref{eq_Carleman_lap} $m$ times, $m\ge 2$, we get the following Carleman estimate for the polyharmonic operator,
\begin{equation}
\label{eq_Carleman_poly}
\|e^{\frac{\varphi}{h}}(-h^2\Delta)^m e^{-\frac{\varphi}{h}}u\|_{H^s_{\textrm{scl}}}\ge \frac{h^m}{C_{s,\Omega}}\|u\|_{H^{s+m}_{\textrm{scl}}},
\end{equation}
for all $u\in C^\infty_0(\Omega)$, $s\in\R$ and $h>0$ small. 
Since we are dealing with first order perturbations of the polyharmonic operator, the following weakened version of \eqref{eq_Carleman_poly}  will be sufficient for our purposes,
\begin{equation}
\label{eq_Carleman_poly_1}
\|e^{\frac{\varphi}{h}}(-h^2\Delta)^m e^{-\frac{\varphi}{h}}u\|_{H^s_{\textrm{scl}}}\ge \frac{h^m}{C_{s,\Omega}}\|u\|_{H^{s+1}_{\textrm{scl}}},
\end{equation}
for all $u\in C^\infty_0(\Omega)$, $s\in\R$ and all $h>0$ small enough.

To add the perturbation $h^{2m}q$ to the estimate \eqref{eq_Carleman_poly_1}, we assume that 
$-1\le s\le 0$ and use that 
\[
\|qu\|_{H^s_{\textrm{scl}}}\le \|qu\|_{L^2}\le \|q\|_{L^\infty}\|u\|_{L^2}\le  \|q\|_{L^\infty}\|u\|_{H^{s+1}_{\textrm{scl}}}.
\] 
To add the perturbation 
\[
h^{2m-1}e^{\frac{\varphi}{h}}(A\cdot hD) e^{-\frac{\varphi}{h}}=h^{2m-1}(A\cdot hD+iA\cdot\nabla \varphi) 
\]
 to the estimate \eqref{eq_Carleman_poly_1}, assuming that 
$-1\le s\le 0$, we need the following estimates 
\[
\|(A\cdot\nabla \varphi)u\|_{H^s_{\textrm{scl}}}\le \|A\cdot\nabla \varphi\|_{L^\infty}\|u\|_{H^{s+1}_{\textrm{scl}}},
\]
\begin{align*}
\|A\cdot hD u\|_{H^s_{\textrm{scl}}}&\le \sum_{j=1}^n\|hD_j(A_j u)\|_{H^s_{\textrm{scl}}}+\mathcal{O}(h)\|(\div A)u\|_{H^s_{\textrm{scl}}}\\
&\le \mathcal{O}(1)\sum_{j=1}^n\|A_j u\|_{H^{s+1}_{\textrm{scl}}}+\mathcal{O}(h)\|u\|_{H^{s+1}_{\textrm{scl}}}\le \mathcal{O}(1)\|u\|_{H^{s+1}_{\textrm{scl}}}.
\end{align*}
When obtaining the last inequality, we notice that by complex interpolation it suffices to consider the cases $s=0$ and $s=-1$. 

Let 
\[
\mathcal{L}_\varphi=e^{\frac{\varphi}{h}} h^{2m}\mathcal{L}_{A,q} e^{-\frac{\varphi}{h}}.
\]
Thus, we obtain the following Carleman estimate for a first order perturbation of the polyharmonic operator.

\begin{prop} 

Let $A\in W^{1,\infty}(\Omega,\C^n)$,  $q\in L^\infty(\Omega,\C)$, and $\varphi$ be a limiting Carleman weight for the semiclassical Laplacian on $\tilde \Omega$. Assume that $m\ge 2$.  If $-1\le s\le 0$, then for $h>0$ small enough,  one 
has 
\begin{equation}
\label{eq_Carleman_poly_perturbation}
\|\mathcal{L}_\varphi u\|_{H^s_{\emph{scl}}}\ge \frac{h^m}{C_{s,\Omega,A,q}}\|u\|_{H^{s+1}_{\emph{scl}}},
\end{equation}
for all $u\in C^\infty_0(\Omega)$. 
\end{prop}

The formal $L^2$-adjoint of $\mathcal{L}_\varphi$ is given by
\[
\mathcal{L}_\varphi^* =e^{-\frac{\varphi}{h}} (h^{2m}\mathcal{L}_{\bar{A},i^{-1}\nabla\cdot \bar{A}+\bar q}) e^{\frac{\varphi}{h}}. 
\]
Notice that  if $\varphi$ is a limiting Carleman weight, then so is $-\varphi$. This implies that the Carleman estimate \eqref{eq_Carleman_poly_perturbation} holds also for the formal adjoint $\mathcal{L}_\varphi^*$.

To construct complex geometric optics solutions we need the following solvability result, similar to \cite{DKSU_2007}.  The proof is essentially well-known, and is included here for the convenience of the reader. In what follows,  we denote by $H^1_{\textrm{scl}}(\Omega)$ the semiclassical Sobolev space of order one on $\Omega$, equipped with the norm
\[
\|u\|_{H^1_{\textrm{scl}}(\Omega)}^2=\|u\|_{L^2(\Omega)}^2+ \|h\nabla u\|_{L^2(\Omega)}^2. 
\]

\begin{prop}
\label{prop_Hahn-Banach}
Let $A\in W^{1,\infty}(\Omega,\C^n)$,  $q\in L^\infty(\Omega,\C)$, and $\varphi$ be a limiting Carleman weight for the semiclassical  Laplacian on $\tilde \Omega$. Assume that $m\ge 2$.
If $h>0$ is small enough, then for any $v\in L^2(\Omega)$, there is a solution $u\in H^1(\Omega)$ of the equation
\[
\mathcal{L}_\varphi u=v \quad \text{in} \quad \Omega,
\]
which satisfies 
\[
\|u\|_{H^1_{\emph{\textrm{scl}}}}\le \frac{C}{h^m}\|v\|_{L^2}.
\] 

\end{prop}

\begin{proof}
Consider the following complex linear functional 
\[
L: \mathcal{L}_\varphi^* C^\infty_0(\Omega)\to \C, \quad \mathcal{L}_\varphi^* w\mapsto (w,v)_{L^2}. 
\]
By the Carleman estimate \eqref{eq_Carleman_poly_perturbation} for the formal adjoint $\mathcal{L}_\varphi^*$, the map $L$ is well-defined. 
Let $w\in  C^\infty_0(\Omega)$. We have
\[
|L(\mathcal{L}_\varphi^* w)|=|(w,v)_{L^2}|\le \|w\|_{L^2} \|v\|_{L^2}\le \frac{C}{h^m}\|\mathcal{L}_\varphi^* w\|_{H^{-1}_{\textrm{scl}}} \|v\|_{L^2},
\]
showing that $L$ is bounded in the $H^{-1}$-norm. Thus, by the Hahn-Banach theorem, we may extend $L$ to a linear continuous functional $\tilde L$ on $H^{-1}(\R^n)$ without increasing the norm. By the Riesz representation theorem, there exists $u\in H^1(\R^n)$ such that for all $w\in H^{-1}(\R^n)$, 
\[
\tilde L(w)=(w,u)_{(H^{-1},H^1)}, \quad \textrm{and}\quad \|u\|_{H^{1}_{\textrm{scl}}}\le \frac{C}{h^m}\|v\|_{L^2}. 
\]
Here $(\cdot,\cdot)_{(H^{-1},H^1)}$ stands for the usual $L^2$-duality. It follows that $\mathcal{L}_\varphi u=v$ in $\Omega$. This completes the proof. 
\end{proof}

Our next goal is to construct complex geometric optics solutions of  the equation $\mathcal{L}_{A,q}u=0$ in $\Omega$, i.e. solutions of the following form,
\[
u(x,\zeta;h)=e^{\frac{ix\cdot \zeta}{h}} (a(x,\zeta)+h r(x,\zeta; h)),
\]  
where $\zeta\in \C^n$ such that $\zeta\cdot \zeta=0$, $|\Re \zeta|=|\Im\zeta|=1$, the amplitude $a\in C^\infty(\overline{\Omega})$, and the remainder satisfies $\|r\|_{H^1_{\textrm{scl}}(\Omega)}= \mathcal{O}(1)$. 

Consider 
\begin{equation}
\label{eq_indep_h}
e^{\frac{-ix\cdot \zeta}{h}} h^{2m}\mathcal{L}_{A,q} e^{\frac{ix\cdot \zeta}{h}}=(-h^2\Delta-2i\zeta\cdot h\nabla)^m+ h^{2m-1}A\cdot hD+h^{2m-1}A\cdot \zeta +h^{2m}q.  
\end{equation}
Since $m\ge 2$, in order Êto get
\begin{equation}
\label{eq_O(h)}
e^{\frac{-ix\cdot \zeta}{h}} h^{2m}\mathcal{L}_{A,q} (e^{\frac{ix\cdot \zeta}{h}} a)=\mathcal{O}(h^{m+1}),
\end{equation}
in $L^2(\Omega)$, 
we should choose $a\in C^\infty(\overline{\Omega})$, satisfying the following first transport equation, 
\begin{equation}
\label{eq_transport}
(\zeta\cdot \nabla)^m a=0\quad \textrm{in}\quad \Omega.  
\end{equation} This is clearly possible. 
Having chosen the amplitude $a$ in this way, we obtain the following equation for $r$,
\begin{equation}
\label{eq_for_r}
e^{\frac{x \cdot\mathrm{Im}\zeta}{h}} h^{2m} \mathcal{L}_{A,q} e^{-\frac{x\cdot \mathrm{Im} \zeta}{h}} e^{\frac{ix\cdot \mathrm{Re}\zeta}{h}} hr=-
e^{\frac{ix\cdot\mathrm{Re}\zeta}{h }}e^{\frac{-ix\cdot \zeta}{h}} h^{2m}\mathcal{L}_{A,q} (e^{\frac{ix\cdot \zeta}{h}} a)\quad \textrm{in }\Omega.
\end{equation}
Thanks to 
Proposition \ref{prop_Hahn-Banach} and \eqref{eq_O(h)}, for $h>0$ small enough,  there exists a solution $r\in H^1(\Omega)$ of \eqref{eq_for_r} such that $\|r\|_{H^1_{\textrm{scl}}}=\mathcal{O}(1)$. 

Summing up, we have the following result.

\begin{prop}

\label{prop_complex_geom_optics}

Let $A\in W^{1,\infty}(\Omega,\C^n)$,  $q\in L^\infty(\Omega,\C)$, and $\zeta\in \C^n$ be such that $\zeta\cdot \zeta=0$ and $|\emph{\Re} \zeta|=|\emph{\Im}\zeta|=1$. Then for all $h>0$ small enough,  there exist solutions $u(x,\zeta; h)\in H^1(\Omega)$ to the equation $\mathcal{L}_{A,q}u=0$ in $\Omega$, of the form
\[
u(x,\zeta;h)=e^{\frac{ix\cdot \zeta}{h}} (a(x,\zeta)+h r(x,\zeta;h)),
\]  
where $a(\cdot,\zeta)\in C^\infty(\overline{\Omega})$ satisfies \eqref{eq_transport} and $\|r\|_{H^1_{\emph{\textrm{scl}}}}=\mathcal{O}(1)$. 

\end{prop}

\begin{rem}
\label{rem_com_geom_1} In what follows, we shall need complex geometric optics solutions belonging to $H^{2m}(\Omega)$. To obtain such solutions, let $\Omega'\supset\supset\Omega$ be a bounded domain with smooth boundary,  and let us extend $A\in W^{1,\infty}(\Omega,\C^n)$ and $q\in L^\infty(\Omega)$ to $W^{1,\infty}(\Omega',\C^n)$ and $L^\infty(\Omega')$-functions, respectively. By elliptic regularity, the complex geometric optics solutions, constructed on $\Omega'$, according to Proposition \ref{prop_complex_geom_optics},   belong to  $H^{2m}(\Omega)$.

\end{rem}

\begin{rem} 
\label{rem_com_geom_2}
We shall also  consider the complex phases $x\cdot \zeta$, with $\zeta$ depending slightly on $h$, to be precise, such that  $\zeta=\zeta^{(0)}+\mathcal{O}(h)$ with $\zeta^{(0)}\in \C^n$ being independent of  $h$.  In this case  it follows from \eqref{eq_indep_h} that we can construct complex geometric optics solutions  with amplitudes $a=a(x,\zeta^{(0)})$, which are independent of $h$ and which satisfy the following transport equation
\[
(\zeta^{(0)}\cdot \nabla)^m a=0\quad \textrm{in}\quad \Omega.  
\]
\end{rem}

\section{Proof of Theorem \ref{thm_main}}

\label{sec_thm1}

The first step is a standard reduction to a larger domain, see \cite{Syl_Uhl_1987}.

\begin{prop}
\label{prop_extension_larger_dom}
Let $\Omega\subset\subset \Omega'$ be two bounded domains in $\R^n$ with smooth boundaries,  and let $A^{(1)},A^{(2)}\in W^{1,\infty}(\Omega',\C^n)$,  $q^{(1)},q^{(2)}\in L^\infty(\Omega',\C)$
satisfy $A^{(1)}=A^{(2)}$ and $q^{(1)}=q^{(2)}$ in $\Omega'\setminus\Omega$. If $\mathcal{C}_{A^{(1)},q^{(1)}}^\Omega=\mathcal{C}_{A^{(2)},q^{(2)}}^\Omega$, then  $\mathcal{C}_{A^{(1)},q^{(1)}}^{\Omega'}=\mathcal{C}_{A^{(2)},q^{(2)}}^{\Omega'}$.

\end{prop}

\begin{proof}

Let $u'\in H^{2m}(\Omega')$ be a solution of $\mathcal{L}_{A^{(1)},q^{(1)}}u'=0$ in $\Omega'$.  
Since $\mathcal{C}_{A^{(1)},q^{(1)}}^\Omega=\mathcal{C}_{A^{(2)},q^{(2)}}^\Omega$, there exists $v\in H^{2m}(\Omega)$, solving  $\mathcal{L}_{A^{(2)},q^{(2)}}v=0$ in $\Omega$, and satisfying $\gamma v=\gamma u'$ in $\p \Omega$ and $\tilde \gamma v=\tilde \gamma u'$ in $\p \Omega$.  Setting
\[
v'=\begin{cases} v & \textrm{in } \Omega,\\
u' & \textrm{in } \Omega'\setminus \Omega,
\end{cases}
\]
we get $v'\in H^{2m}(\Omega')$ and $\mathcal{L}_{A^{(2)},q^{(2)}}v'=0$ in $\Omega'$. Thus, $\mathcal{C}_{A^{(1)},q^{(1)}}^{\Omega'}\subset\mathcal{C}_{A^{(2)},q^{(2)}}^{\Omega'}$. The same argument in the other direction shows the claim.

\end{proof}

Recall that $A^{(j)}\in W^{1,\infty}(\R^n,\C^n)\cap \mathcal{E}'(\overline\Omega,\C^n)$, $j=1,2$.  Let $B(0,R)$ be a open ball in $\R^n$, centered at the origin and of  radius $R$ such that  $\Omega\subset\subset B(0,R)$. 
We extend $q^{(j)}\in L^\infty(\Omega)$, $j=1,2$, by zero to $B(0,R)\setminus \Omega$, and denote these extensions by the same letters.  
According to Proposition \ref{prop_extension_larger_dom}, we know that $\mathcal{C}_{A^{(1)},q^{(1)}}^{B(0,R)}=\mathcal{C}_{A^{(2)},q^{(2)}}^{B(0,R)}$. 

We shall need the following consequence of the Green's formula, see \cite{Agmon_book}, 
\begin{equation}
\label{eq_Green_form}
(\mathcal{L}_{A,q}u,v)_{L^2(B(0,R))} = (u, \mathcal{L}_{A,q}^*v)_{L^2(B(0,R))},\quad u,v\in H^{2m}(B(0,R)),  \gamma u=\tilde \gamma u=0,
\end{equation}
where $\mathcal{L}_{A,q}^*=\mathcal{L}_{\bar{A},i^{-1}\nabla \cdot \bar{A}+\bar q}$.

Let $u_1\in H^{2m}(B(0,R))$ be a solution to $\mathcal{L}_{A^{(1)},q^{(1)}} u_1=0$ in $B(0,R)$. Then there exists a solution $u_2\in H^{2m}(B(0,R))$ to 
$\mathcal{L}_{A^{(2)},q^{(2)}} u_2=0$ in $B(0,R)$ such that $\gamma u_1=\gamma u_2$ and $\tilde \gamma u_1=\tilde \gamma u_2$.  
We have
\[
\mathcal{L}_{A^{(1)},q^{(1)}}(u_1-u_2)=(A^{(2)}-A^{(1)})\cdot Du_2 + (q^{(2)}-q^{(1)})u_2\quad \textrm{in}\quad  B(0,R). 
\]
Let $v\in H^{2m}(B(0,R))$ satisfy 
\[
\mathcal{L}_{A^{(1)},q^{(1)}}^* v=0\quad \textrm{in}\quad B(0,R). 
\]
Using \eqref{eq_Green_form}, we get
\begin{equation}
\label{eq_identity_main}
\int_{B(0,R)} ((A^{(2)}-A^{(1)})\cdot Du_2)\bar v dx  + \int_{B(0,R)} (q^{(2)}-q^{(1)})u_2 \bar vdx=0. 
\end{equation}

To show the equalities $A^{(1)}=A^{(2)}$ and $q^{(1)}=q^{(2)}$, the idea is to use the identity \eqref{eq_identity_main} with $u_2$ and $v$ being complex geometric optics solutions. 
To construct these solutions, let $\xi,\mu_1,\mu_2 \in \R^n$ be such that $|\mu_1|=|\mu_2|=1$ and $\mu_1\cdot\mu_2=\mu_1\cdot \xi=\mu_2\cdot \xi=0$. 
Similarly to \cite{Sun_1993}, we set for $h>0$ small enough, 
\begin{equation}
\label{eq_choice_zeta}
\zeta_2=\frac{h\xi}{2}+\sqrt{1-h^2\frac{|\xi|^2}{4}}\mu_1+i\mu_2,\quad 
\zeta_1=-\frac{h\xi}{2}+\sqrt{1-h^2\frac{|\xi|^2}{4}}\mu_1-i\mu_2,
\end{equation}
so that $\zeta_j\cdot \zeta_j=0$, $j=1,2$, and $\zeta_2-\bar\zeta_1=h\xi$.

Then by Proposition \ref{prop_complex_geom_optics}, Remarks \ref{rem_com_geom_1} and \ref{rem_com_geom_2}, for $h>0$ small enough,  there exist solutions $u_2(x,\zeta_2;h)\in H^{2m}(B(0,R))$ and $v(x,\zeta_1;h)\in H^{2m}(B(0,R))$, to the equations 
\[
\mathcal{L}_{A^{(2)},q^{(2)}}u_2=0\quad \textrm{in } B(0,R)\quad\textrm{and}\quad
\mathcal{L}_{A^{(1)},q^{(1)}}^*v=0\quad \textrm{in } B(0,R), 
\]
respectively, of the form
\begin{equation}
\label{eq_u_2_v_optics}
\begin{aligned}
u_2(x,\zeta_2;h)=&e^{\frac{ix\cdot \zeta_2}{h}} (a_2(x,\mu_1+i\mu_2)+h r_2(x,\zeta_2;h)),\\
v(x,\zeta_2;h)=&e^{\frac{ix\cdot \zeta_1}{h}} (a_1(x,\mu_1-i\mu_2)+h r_1(x,\zeta_1;h)),
\end{aligned}
\end{equation}
where the amplitudes $a_1(\cdot,\mu_1+i\mu_2),a_2(\cdot,\mu_1-i\mu_2)\in C^\infty(\overline{B(0,R)})$ satisfy the transport equations, 
\begin{equation}
\label{eq_transp_new_1}
((\mu_1+i\mu_2)\cdot \nabla)^m a_2(x,\mu_1+i\mu_2)=0,\   
((\mu_1-i\mu_2)\cdot \nabla)^m a_1(x,\mu_1-i\mu_2)=0\ \textrm{in}\ B(0,R),
\end{equation}
and 
\begin{equation}
\label{eq_remainder_r_j}
\|r_j\|_{H^1_{\textrm{scl}}}=\mathcal{O}(1),\quad j=1,2.
\end{equation}  

Substituting $u_2$ and $v$, given by \eqref{eq_u_2_v_optics}, in \eqref{eq_identity_main}, we get
\begin{equation}
\label{eq_identity_main_obtics}
\begin{aligned}
0=&\int_{B(0,R)}(A^{(2)}-A^{(1)})\cdot \frac{\zeta_2}{h}e^{i x\cdot\xi}(a_2+hr_2)(\bar a_1+h\bar r_1)dx\\
 &+ 
\int_{B(0,R)}(A^{(2)}-A^{(1)})\cdot e^{ix\cdot\xi} (Da_2+hDr_2)(\bar a_1+ h\bar r_1)dx \\
&+\int_{B(0,R)}(q^{(2)}-q^{(1)})e^{ix\cdot \xi}(a_2+hr_2)(\bar a_1+h\bar r_1)dx.
\end{aligned}
\end{equation}
Multiplying \eqref{eq_identity_main_obtics} by $h$ and letting $h\to +0$, we obtain that 
\begin{equation}
\label{eq_recover_A}
(\mu_1+i\mu_2)\cdot \int_{B(0,R)} (A^{(2)}(x)-A^{(1)}(x)) e^{ix\cdot\xi} a_2(x,\mu_1+i\mu_2)\overline{a_1(x,\mu_1-i\mu_2)}dx=0.  
\end{equation}
Here we use \eqref{eq_remainder_r_j} and the fact that 
$a_j\in C^\infty(\overline{B(0,R)})$, $j=1,2$, to conclude that
\begin{align*}
&\bigg|\int_{B(0,R)} r_2\bar r_1dx\bigg|\le \|r_2\|_{L^2}\|r_1\|_{L^2}\le \mathcal{O}(1),\quad
\bigg|\int_{B(0,R)} a_2 \bar r_1dx\bigg|\le \mathcal{O}(1),\\
&\bigg|\int_{B(0,R)} (hD r_2) \bar a_1 dx\bigg|\le \mathcal{O}(1)\|hD r_2\|_{L^2}\le \mathcal{O}(1). 
\end{align*}

Substituting $a_1=a_2=1$ in  \eqref{eq_recover_A}, we have
\begin{equation}
\label{eq_recover_rot}
(\mu_1+i\mu_2)\cdot \int_{B(0,R)} (A^{(2)}-A^{(1)}) e^{ix\cdot\xi} dx=0.
\end{equation}

Similarly to \cite{Salo_diss},  
we conclude from \eqref{eq_recover_rot} that 
\begin{equation}
\label{eq_recover_rot_1}
\p_j(A_k^{(2)}-A_k^{(1)})-\p_k(A_j^{(2)}-A_j^{(1)})=0\quad\textrm{in } B(0,R),\quad 1\le j,k\le n.
\end{equation}
For the convenience of the reader, we recall the arguments of \cite{Salo_diss}. Indeed,
\eqref{eq_recover_rot}  implies that
\begin{equation}
\label{eq_recover_rot_2}
\mu\cdot (\hat{A^{(2)}}(\xi)-\hat{A^{(1)}}(\xi))=0\quad\textrm{for all }\mu, \xi \in\R^n,\ \mu\cdot\xi=0,
\end{equation}
where $\hat{A^{(j)}}$ stands for the Fourier transform of $A^{(j)}$.  Let $\xi=(\xi_1,\dots,\xi_n)$ and for $j\ne k$, $1\le j,k\le n$, consider the vectors $\mu=\mu(\xi,j,k)$ such that $\mu_j=-\xi_k$, $\mu_k=\xi_j$ and all other components of $\mu$ are equal to zero. Thus, $\mu\cdot\xi=0$ and \eqref{eq_recover_rot_2} yields that
\[
\xi_j\cdot (\hat{A_k^{(2)}}(\xi)-\hat{A_k^{(1)}}(\xi))-\xi_k\cdot (\hat{A_j^{(2)}}(\xi)-\hat{A_j^{(1)}}(\xi))=0. 
\]
This proves \eqref{eq_recover_rot_1}. 

It follows from \eqref{eq_recover_rot_1} that the function $\varphi\in C^{1}(\overline{B(0,R)})$, given by
\[
\varphi(x)=\int_0^1 (A^{(2)}(tx)-A^{(1)}(tx))\cdot x dt,
\]
satisfies
\[
A^{(2)}-A^{(1)}=\nabla \varphi \quad \textrm{in}\quad B(0,R).
\]

Since $A^{(2)}-A^{(1)}=0$ in a neighborhood of  the boundary $\p B(0,R)$, we conclude that $\varphi$ is  a constant, say $c\in \C$, on $\p B(0,R)$. Thus, considering $\varphi-c$, we may and shall assume that $\varphi=0$ on $\p B(0,R)$.

Let us now show that $A^{(1)}=A^{(2)}$. To that end, consider \eqref{eq_recover_A} with 
$a_1(x,\mu_1-i\mu_2)=1$ and $a_2(x,\mu_1+i\mu_2)$ satisfying  
\begin{equation}
\label{eq_di_bar_1}
((\mu_1+i\mu_2)\cdot \nabla) a_2(x,\mu_1+i\mu_2)=1\quad \textrm{in} \quad B(0,R). 
\end{equation}
Notice that such a choice is possible thanks to \eqref{eq_transp_new_1}.  
Here we also remark that  \eqref{eq_di_bar_1} is an inhomogeneous $\bar\p$-equation and we  may solve it   by setting
\[
a_2(x,\mu_1+i\mu_2)=\frac{1}{2\pi}\int_{\R^2}\frac{1}{y_1+i y_2}\chi(x-y_1\mu_1-y_2\mu_2)dy_1dy_2,
\]
where $\chi\in C^\infty_0(\R^n)$ is such that $\chi=1$ near $\overline{B(0,R)}$.

We have from \eqref{eq_recover_A},
\[
(\mu_1+i\mu_2)\cdot \int_{B(0,R)} (\nabla\varphi(x))  e^{ix\cdot\xi} a_2(x,\mu_1+i\mu_2) dx=0.  
\]
Integrating by parts and using the facts that  $\varphi =0$ on $\p B(0,R)$ and $\mu_1\cdot\xi=\mu_2\cdot\xi=0$, we obtain that
\[
0=\int_{B(0,R)} \varphi(x)  e^{ix\cdot\xi} ((\mu_1+i\mu_2)\cdot\nabla) a_2(x,\mu_1+i\mu_2) dx=\int_{B(0,R)} \varphi(x)  e^{ix\cdot\xi}dx. 
\]
Thus, $\varphi=0$ in $B(0,R)$, and therefore, $A^{(1)}=A^{(2)}$.

It is now easy to show that $q^{(1)}=q^{(2)}$. To this end, we substitute $A^{(1)}=A^{(2)}$ and $a_1=a_2=1$ in the identity \eqref{eq_identity_main_obtics} and obtain that 
\[
\int_{B(0,R)}(q^{(2)}-q^{(1)})e^{ix\cdot \xi}(1+hr_2)(1+h\bar r_1)dx=0.
\]
Letting $h\to+0$, we get $\hat{q^{(2)}}(\xi)=\hat{q^{(1)}}(\xi)$, for all $\xi\in \R^n$, and therefore, $q^{(1)}=q^{(2)}$ in $B(0,R)$. 
This completes the proof of Theorem \ref{thm_main}.

 \section{Boundary reconstruction and proof of Theorem \ref{thm_main_2} }
 
The proof of Theorem \ref{thm_main_2} will be obtained by reducing it to Theorem \ref{thm_main}, with the help of the following boundary determination result. 
 \begin{prop}
 \label{prop_boundary_reconstruction}
 Let $\Omega\subset \R^n$, $n\ge 3$,  be a bounded domain with  $C^\infty$-boundary, and let $A\in C^{\infty}(\overline{\Omega},\C^n)$ and $q\in C^\infty(\overline \Omega,\C)$.  
The knowledge of the set of the Cauchy data  $\mathcal{C}_{A,q}$ determines the values of $A$ on $\p \Omega$.  
\end{prop}

 When proving Proposition \ref{prop_boundary_reconstruction}, it will be convenient to rewrite the equation
 \begin{equation}
 \label{eq_L_A_bound}
 \mathcal{L}_{A,q}u=((-\Delta)^m+A\cdot D+q)u=0\quad\textrm{in}\quad \Omega,\quad m\ge 2,
 \end{equation}
 as a  second order system. 
 Introducing 
 \[
 u_1=u,\ u_2=(-\Delta) u,\ \dots,\ u_m=(-\Delta)^{m-1}u,
 \]
 we get
 \begin{equation}
 \label{eq_system_lap}
 ( -\Delta\otimes I +R_1(x,D_{x})+R_0(x))U=0 \quad\textrm{in } \Omega,
 \end{equation}
 where $I$ is the $m\times m$-identity matrix,  $U=(u_1,u_2,\dots, u_m)^{t}$, and
 \[
 R_1(x,D_{x})=\begin{pmatrix} 0& 0 & \hdots & 0\\
 \vdots& \vdots& \hdots& \vdots\\
 0 & 0 & \hdots & 0\\
 A(x)\cdot D_{x} & 0 & \hdots & 0
 \end{pmatrix},
 R_0(x)=\begin{pmatrix} 0 & -1 & 0 & \hdots & 0\\
 0 & 0 & -1 & \hdots & 0\\
 \vdots & \vdots & \vdots & \hdots & \vdots\\
 0 & 0 & 0 & \hdots & -1\\
 q(x) & 0 & 0 & \hdots & 0
 \end{pmatrix}.
 \]
 The set of the Cauchy data for the system \eqref{eq_system_lap} is defined as  follows, 
 \[
\{ (U|_{\p \Omega}, \p_\nu U|_{\p \Omega}): U\in (H^2(\Omega))^m,  \ U\textrm{ solves } \eqref{eq_system_lap} \}.
 \]
 
 When considering the system \eqref{eq_system_lap}  near the boundary, we shall make use of the boundary normal coordinates.  
 Let $(x^1,\dots,  x^n)$ be the boundary normal
coordinates defined locally near a point at the boundary. Here $x'=(x^1,\dots, x^{n-1})$ is a local coordinate for $\p \Omega$ and $x_n\ge 0$ is the distance to the boundary. In these coordinates, the Euclidean metric has the form, see \cite{LeeUhl89}, 
\[
g=\sum_{j,k=1}^{n-1}g_{jk}(x)dx^j dx^k +(dx^n)^2,
\]
and  the Euclidean Laplacian is given by
\[
-\Delta=D_{x^n}^2+i E(x)D_{x^n} +Q(x,D_{x'}).
\]
Here $E\in C^\infty$ and $Q(x,D_{x'})$ is a second order differential operator in $D_{x'}$, depending smoothly on $x^n\ge 0$. The principal part $Q_2$ of $Q$ is given by
\[
Q_2(x,D_{x'})=\sum_{j,k=1}^{n-1}g^{jk}(x)D_{x^j}D_{x^k}.
\]
  Here $(g^{jk})$ is the inverse of the matrix $(g_{jk})$.  We remark that here $E(x)$ and $Q(x,D_{x'})$ are known. 
  It follows therefore that  the set of the Cauchy data for the system 
 \eqref{eq_system_lap} and the set of the Cauchy data for the equation \eqref{eq_L_A_bound} coincide.

In the boundary normal coordinates, the operator in the system \eqref{eq_system_lap} has the form,
 \begin{equation}
 \label{eq_system_b_n}
 P(x,D):=(D_{x^n}^2+i E(x)D_{x^n} +Q(x,D_{x'}))\otimes I+ \tilde R_1^{(1)}(x,D_{x'})+\tilde R_1^{(2)}(x) D_{x^n}+\tilde R_0(x).
 \end{equation}
The matrix $\tilde R_0$ is just $R_0$, expressed in the boundary normal coordinates, $\tilde R_1^{(1)}$ and $\tilde R_1^{(2)}$ are $m\times m$ matrices all of whose entries are zero, except for the entry $(m,1)$, which is equal to $\sum_{j=1}^{n-1}\tilde A_jD_{x^j}$ and $\tilde A_n$, respectively. The functions $\tilde A_j$, $j=1,\dots,n$, are the components of the vector field $A$ in the boundary normal coordinates.

 We have the following result, which is a direct analog of  \cite{LeeUhl89, NakSunUlm_1995}.  
  
 \begin{prop}
 \label{prop_factorization}
 There is a matrix-valued pseudodifferential operator $B(x, D_{x'})$ of order one in $x'$ depending smoothly on $x^n$ such that 
\begin{equation}
\label{eq_factor}
P(x,D)=(D_{x^n}\otimes I + iE(x)\otimes I + \tilde R_1^{(2)}(x)- i B(x,D_{x'}) )(D_{x^n}\otimes I+ iB(x,D_{x'})),
\end{equation}
modulo a smoothing operator. Here $B(x,D_{x'})$ is unique modulo a smoothing term, if we require that its principal symbol is given by
$-\sqrt{Q_2(x,\xi')}I $. 

\end{prop}

\begin{proof}

Combining \eqref{eq_system_b_n} and \eqref{eq_factor}, we see that we should have
\begin{equation}
\label{eq_factor_2}
\begin{aligned}
B^2(x,D_{x'})+i[D_{x^n}\otimes I,B(x,D_{x'})]-E(x)B(x,D_{x'})+i\tilde R^{(2)}_1(x)B(x,D_{x'})\\
=Q(x,D_{x'})\otimes I +\tilde R^{(1)}_1(x,D_{x'})+\tilde R_0(x),
\end{aligned}
\end{equation}
modulo a smoothing operator. Let us write the full symbol of $B(x,D_{x'})$ as follows,
\[
B(x,\xi')\sim \sum_{j=-\infty}^1 b_j(x,\xi'),
\]
with $b_j$ taking values in $m\times m$ matrices with entries homogeneous of degree $j$ in $\xi'=(\xi_1,\dots,\xi_{n-1})\in\R^{n-1}$. Thus, 
\eqref{eq_factor_2}  implies that
\begin{align}
\label{eq_lap_4}
\sum_{l=-\infty}^2\underset{|\alpha|\ge 0,\ j,k\le 1}{(\sum_{j+k-|\alpha|=l}}\frac{1}{\alpha!}\p_{\xi'}^\alpha b_jD_{x'}^\alpha b_k)
-E(x)\sum_{j=-\infty}^1 b_j(x,\xi')+\sum_{j=-\infty}^1 \p_{x^n}b_j(x,\xi')\\ \nonumber
+i \tilde R^{(2)}_1(x)\sum_{j=-\infty}^1b_j(x,\xi')
=
Q_2(x,\xi') I+Q_1(x,\xi') I + \tilde R^{(1)}_1(x,\xi')+\tilde R_0(x).
\end{align}
Here $Q_1(x,\xi')=Q(x,\xi')-Q_2(x,\xi')$ is homogeneous of degree one  in $\xi'$. 
Equating the terms homogeneous of degree two in \eqref{eq_lap_4}, we get
\[
b_1^2(x,\xi')=Q_2(x,\xi')I,
\] 
so we should choose $b_1(x,\xi')$
to be the scalar matrix given by 
\begin{equation}
\label{eq_b_1}
b_1(x,\xi')=-\sqrt{Q_2(x,\xi')}I. 
\end{equation}
Equating the terms homogeneous of degree one in \eqref{eq_lap_4},  we have an equation for $b_0(x,\xi')$,
\begin{equation}
\label{eq_b_0}
2b_1b_0+\sum_{|\alpha|=1}\p_{\xi'}^\alpha b_1D_{x'}^\alpha b_1- Eb_1+i \tilde R^{(2)}_1b_1+\p_{x_n}b_1=Q_1(x,\xi')+\tilde R^{(1)}_1(x,\xi'),
\end{equation}
which has a unique solution. 
The terms $b_j$, $j\le -1$, are chosen in a similar fashion, by equating terms of degree of homogeneity $j+1$ in   \eqref{eq_lap_4}. This completes the proof.

\end{proof}

Using Proposition \ref{prop_factorization} and the same argument as in the proof of \cite[Proposition 1.2]{LeeUhl89}, we conclude that the set of the Cauchy data for the system  \eqref{eq_system_lap} determines the operator $B((x',0),D_{x'})$ modulo smoothing. It follows therefore, from \eqref{eq_b_1} and  \eqref{eq_b_0} that the set of the Cauchy data determines the expressions,
\[
i\tilde A_n(x',0)\sqrt{Q_2((x',0),\xi')}+\sum_{j=1}^{n-1}\tilde A_j(x',0)\xi_j,\quad \xi'\in \R^{n-1}.  
\]
This completes the proof of Proposition \ref{prop_boundary_reconstruction}.

Let $\tilde \Omega$ be a bounded domain in $\R^n$ with $C^\infty$-boundary such that $\Omega\subset\subset \tilde\Omega$.  Then by Proposition \ref{prop_boundary_reconstruction} we can extend $A^{(j)}$, $j=1,2$, to compactly supported vector fields  in $W^{1,\infty}(\tilde \Omega, \C^n)$ such that $A^{(1)}=A^{(2)}$ in $\tilde \Omega\setminus\Omega$.  Let us also extend $q^{(j)}$, $j=1,2$, to $\tilde \Omega$ by zero. Theorem \ref{thm_main_2} follows therefore  combining  Proposition
\ref{prop_extension_larger_dom} and Theorem \ref{thm_main}.

\section{Proof of Theorem \ref{thm_main_3}} 

Let $\Omega\subset\R^n$, $n\ge 3$, be a bounded domain in $\R^n$ with connected $C^\infty$-boundary. In this case, arguing as in  Section \ref{sec_thm1},  
see \eqref{eq_recover_A}, we obtain that 
\begin{equation}
\label{eq_recover_A_conv}
(\mu_1+i\mu_2)\cdot \int_{\Omega} (A^{(2)}(x)-A^{(1)}(x)) e^{ix\cdot\xi} a_2(x,\mu_1+i\mu_2)\overline{a_1(x,\mu_1-i\mu_2)}dx=0,
\end{equation}
for any $\mu_1,\mu_2,\xi\in \R^n$ such that $\mu_1\cdot \xi=\mu_2\cdot\xi=\mu_1\cdot \mu_2=0$ and $|\mu_1|=|\mu_2|=1$. Here 
$a_1(\cdot,\mu_1+i\mu_2),a_2(\cdot,\mu_1-i\mu_2)\in C^\infty(\overline{\Omega},\C)$ satisfy the transport equations, 
\[
((\mu_1+i\mu_2)\cdot \nabla)^m a_2(x,\mu_1+i\mu_2)=0\  \textrm{and}\ 
((\mu_1-i\mu_2)\cdot \nabla)^m a_1(x,\mu_1-i\mu_2)=0\ \textrm{in}\ \Omega.
\]
 We can go back to Section \ref{sec_thm1} and repeat the construction of complex geometric optics solutions, with vectors 
$\zeta_1$ and $\zeta_2$, given by \eqref{eq_choice_zeta}, where $\mu_2$ is replaced by $-\mu_2$.  In this way, instead of \eqref{eq_recover_A_conv} we arrive at
\begin{equation}
\label{eq_recover_A_conv_2}
(\mu_1-i\mu_2)\cdot \int_{\Omega} (A^{(2)}(x)-A^{(1)}(x)) e^{ix\cdot\xi} a_2(x,\mu_1-i\mu_2)\overline{a_1(x,\mu_1+i\mu_2)}dx=0,
\end{equation}
where
$a_1(\cdot,\mu_1-i\mu_2),a_2(\cdot,\mu_1+i\mu_2)\in C^\infty(\overline{\Omega},\C)$ satisfy the transport equations, 
\[
((\mu_1-i\mu_2)\cdot \nabla)^m a_2(x,\mu_1-i\mu_2)=0\  \textrm{and}\ 
((\mu_1+i\mu_2)\cdot \nabla)^m a_1(x,\mu_1+i\mu_2)=0\ \textrm{in}\ \Omega.
\]
Moreover, choosing $a_2=a_1=1$ in \eqref{eq_recover_A_conv} and using Proposition \ref{prop_existence_varphi} in the Appendix, we get
\[
A^{(2)}-A^{(1)}=\nabla \varphi,
\]
where $\varphi\in C^{1}(\overline{\Omega})$. 

Choosing next $a_1(\cdot,\mu_1-i\mu_2)=a_1(\cdot,\mu_1+i\mu_2)=1$ and letting $a_2(\cdot,\mu_1+i\mu_2)$, $a_2(\cdot,\mu_1-i\mu_2)\in C^\infty(\overline{\Omega},\C)$ be such that
\[
((\mu_1+i\mu_2)\cdot \nabla) a_2(x,\mu_1+i\mu_2)=0\  \textrm{and}\ 
((\mu_1-i\mu_2)\cdot \nabla) a_2(x,\mu_1-i\mu_2)=0\ \textrm{in}\ \Omega,
\]
we conclude from \eqref{eq_recover_A_conv}
and \eqref{eq_recover_A_conv_2} that 
\begin{equation}
\label{eq_recover_A_conv_3}
(\mu_1\pm i\mu_2)\cdot \int_{\Omega} (\nabla \varphi) e^{ix\cdot\xi} a_2(x,\mu_1 \pm i\mu_2)dx=0.
\end{equation}

\begin{prop}
The function $\varphi$ is constant along the connected set $\p \Omega$.

\end{prop}

\begin{proof}

Completing the orthonormal family $\mu_1,\mu_2$ to an orthonormal basis in $\R^n$, $\mu_1,\mu_2,\mu_3,\dots,\mu_n$, we have for any vector $x\in\R^n$, 
\[
x=\sum_{j=1}^{n}(x\cdot\mu_j)\mu_j.
\]
We introduce new linear coordinates in $\R^n$, given by the orthogonal transformation $T:\R^n\to \R^n$,  $T(x)=y$, $y_j=x\cdot \mu_j$, $j=1,\dots,n$, and denote 
\[
z=y_1+iy_2,\quad \p _{\bar z}=\frac{1}{2}(\p_{y_1}+i\p_{y_2}),\quad \p _{z}=\frac{1}{2}(\p_{y_1}-i\p_{y_2}).
\]
Thus,
\[
(\mu_1+i\mu_2)\cdot \nabla=2\p_{\bar z},\quad (\mu_1-i\mu_2)\cdot \nabla=2\p_{z}. 
\]
Hence, changing variables in \eqref{eq_recover_A_conv_3}, we get
\begin{equation}
\label{eq_T_omega}
 \int_{T(\Omega)} (\p_{\bar z}\tilde \varphi(y)) e^{iy\cdot \xi} g_1(y)dy=0,\quad \int_{T(\Omega)} (\p_{z}\tilde \varphi(y)) e^{iy\cdot \xi} g_2(y)dy=0,
\end{equation}
for all $\xi=(0,0,\xi'')$, $\xi''\in \R^{n-2}$, and all $g_1,g_2\in C^\infty(\overline{T(\Omega)})$ such that $\p_{\bar z} g_1=0$, $\p_z g_2=0$.  Here $\tilde \varphi(y):=\varphi(T^{-1}y)$. 
Taking the inverse Fourier transform in \eqref{eq_T_omega} with respect to the  variable $\xi''$, we get, for all $y''\in \R^{n-2}$,
\begin{equation}
\label{eq_T_omega_2}
\int_{\Omega_{y''}} (\p_{\bar z}\tilde \varphi) g_1(y)dy_1dy_2=0,\quad \int_{\Omega_{y''}} (\p_{z}\tilde \varphi)  g_2(y)dy_1dy_2=0,
\end{equation}
where $\Omega_{y''}=T(\Omega)\cap \Pi_{y''}$ and $\Pi_{y''}=\{(y_1,y_2,y''):(y_1,y_2)\in\R^2\}$.

Let us show that for almost  all    $y''\in \R^{n-2}$, the boundary of the set $\Omega_{y''}$ is $C^\infty$-smooth. To that end, consider the function  $f=(f_1,\dots, f_{n-2}):\p T(\Omega)\to \R^{n-2}$, given by  $(y_1,y_2,y'')\mapsto y''$. 
Let $\rho\in C^\infty(\R^n, \R)$ be a defining function of  $T(\Omega)$, i.e. 
$T(\Omega)=\{y\in \R^n:\rho(y)<0\}$ and $\nabla \rho(y)\ne 0$ for all $y\in \p T(\Omega)$. Then a point $y_0\in\p T(\Omega)$ is a critical point of the function $f$  precisely when the vectors $\nabla f_1(y_0),\dots,\nabla f_{n-2}(y_0),\nabla \rho(y_0)$ are linearly dependent. 
The latter holds precisely when 
$(\p_{y_1}\rho(y_0),\p_{y_2}\rho(y_0))=0$.
Since $\Omega_{y''}=\{(y_1,y_2):\rho(y_1,y_2,y'')<0\}$, it follows, when $y''$ is not a critical value of $f$, that the boundary of $\Omega_{y''}$ is $C^\infty$-smooth. 
By Sard's theorem, the set of critical values of the function $f$ is of measure zero and therefore,  for almost  all    $y''\in \R^{n-2}$, the boundary of the set $\Omega_{y''}$ is $C^\infty$-smooth.

Using Stokes'  theorem,  it follows from \eqref{eq_T_omega_2} that for almost all $y''$, 
\begin{equation}
\label{eq_T_omega_3}
\int_{\p \Omega_{y''}}\tilde \varphi g_1dz=0,\quad \int_{\p \Omega_{y''}}\tilde \varphi g_2d\bar z=0,
\end{equation}
for any $g_1\in C^\infty(\overline{\Omega_{y''}})$ holomorphic and $g_2\in C^\infty(\overline{\Omega_{y''}})$ antiholomorphic functions in $\Omega_{y''}$. 
In particular, taking $g_2=\bar g_1$ in \eqref{eq_T_omega_3}, we have
\[
\int_{\p \Omega_{y''}}\tilde \varphi \bar g_1d\bar z=0, \textrm{ and therefore, } 
\int_{\p \Omega_{y''}}\overline{\tilde \varphi}  g_1dz=0.
\]
Hence,
\[
\int_{\p \Omega_{y''}}(\Re \tilde \varphi ) g d z=0,\quad \int_{\p \Omega_{y''}}(\Im \tilde\varphi)  g d z=0,
\]
for any holomorphic function $g\in C^ \infty(\overline{\Omega_{y''}})$.

Next we shall show that for almost all $y''$,  $\Re\tilde \varphi$ and $\Im\tilde \varphi$ are constant along the boundary of each connected component of $\Omega_{y''}$. 
Indeed, the fact that 
\begin{equation}
\label{eq_T_omega_4}
\int_{\p \Omega_{y''}}(\Re \tilde \varphi ) g d z=0,
\end{equation}
for any holomorphic function $g\in C^\infty(\overline{\Omega_{y''}})$, implies that there exists a holomorphic function $F\in C(\overline{\Omega_{y''}})$ such that $F|_{\p \Omega_{y''}}=\Re \tilde \varphi|_{\p \Omega_{y''}}$. 
To see this, we shall follow the proof of \cite[Lemma 5.1]{DKSU_2007}, and let  $F$  be the Cauchy integral of $\Re \tilde \varphi|_{\p \Omega_{y''}}$, i.e.,
\[
F(z)=\frac{1}{2\pi i}\int_{\p \Omega_{y''}}\frac{\Re\tilde  \varphi(\zeta)}{\zeta-z}d\zeta,\quad z\in \C\setminus \p \Omega_{y''}.
\]   
The function $F$ is holomorphic inside and outside $\p \Omega_{y''}$. As $\varphi\in C^1$, the Plemelj-Sokhotski-Privalov formula states that
\begin{equation}
\label{eq_T_omega_5}
\lim_{z\to z_0, z\in  \Omega_{y''}} F(z)-\lim_{z\to z_0,z\not\in \Omega_{y''}}F(z)=\Re \tilde  \varphi(z_0),\quad z_0\in \p \Omega_{y''}. 
\end{equation}
The function $\zeta\mapsto (\zeta-z)^{-1}$ is holomorphic on $\Omega_{y''}$ when $z\not\in\Omega_{y''}$. Thus, \eqref{eq_T_omega_4} implies that $F(z)=0$ when $z\not\in\Omega_{y''}$. 
Hence, the second limit in \eqref{eq_T_omega_5} is zero and therefore, $F$ is a holomorphic function on $\Omega_{y''}$ whose restriction to the boundary agrees with $\Re \tilde \varphi$.  

Moreover, $\Delta \Im F=0$ in $\Omega_{y''}$ and $\Im F|_{\p \Omega_{y''}}=0$. Hence, $F$ is real-valued and therefore, is  constant on each connected component of $\Omega_{y''}$.  Hence, $\Re \tilde \varphi$ is constant along the boundary of each connected component of $\Omega_{y''}$,  for almost all $y''$. 
In the same way, one shows that  $\Im\tilde \varphi$Ê is constant along  the boundary of each connected component of $\Omega_{y''}$,  for almost all $y''$.

Going back to the $x$-coordinates, we conclude that the function $\varphi(x)$ is constant  along  the boundary of each connected component of the section $T^{-1}(\Omega_{y''})=\Omega\cap T^{-1}(\Pi_{y''})$,  for almost all $y''\in \R^{n-2}$, where the two-dimensional plane $T^{-1}(\Pi_{y''})$ is given by
\begin{equation}
\label{eq_plane_T}
T^{-1}(\Pi_{y''})=\bigg\{x=y_1\mu_1+y_2\mu_2+\sum_{j=3}^ny_j\mu_j:y_1,y_2\in \R,y''=(y_3,\dots,y_n)-\textrm{fixed}\bigg\}. 
\end{equation}
Here $\mu_1,\mu_2\in \R^n$ are arbitrary vectors such that $\mu_1\cdot\mu_2=0$ and $|\mu_1|=|\mu_2|=1$. 

Let us now prove that $\varphi$ is constant along the boundary $\p \Omega$. The fact that $\varphi$ is constant along the boundary of each connected component of $ T^{-1}(\Omega_{y''})$, for some $y''$, means that $X\varphi=0$ when $X$ is a tangential vector field to  $\p T^{-1}(\Omega_{y''})$.  Let $a\in \p \Omega$. Then to show that $\varphi$ is a constant in a neighborhood of $a$,  one has to establish that  
\[
X_{j}\varphi=0, \quad j=1,\dots,n-1,
\]
 where $X_1,\dots,X_{n-1}$  is a basis of tangential vector fields to $\p \Omega$ in a  neighborhood of $a$. 

Without loss of generality,  we may assume that locally near  $a=(a',a_n)\in \p \Omega$, $a'=(a_1,\dots,a_{n-1})$, the boundary $\p \Omega$ has the form $x_n=\phi(x_1,\dots,x_{n-1})$ where $\phi$  is $C^\infty$ near $a'$.  Then for $x'$ in a neighborhood of $a'$,  the vectors
\begin{equation}
\label{eq_X_j}
X_1(x')=\begin{pmatrix} 1\\
0\\
\vdots\\
0\\
\p_{x_1}\phi(x')
\end{pmatrix},
X_2(x')=\begin{pmatrix} 0\\
1\\
\vdots\\
0\\
\p_{x_2}\phi(x')
\end{pmatrix},\dots,
X_{n-1}(x')=\begin{pmatrix} 0\\
0\\
\vdots\\
1\\
\p_{x_{n-1}}\phi(x')
\end{pmatrix},
\end{equation}
  form a basis of tangential vector fields to $\p \Omega$ in a  neighborhood of $a$.  
 
 Let us show that for any $X_j$, $j=1,\dots,n-1$, given in \eqref{eq_X_j},  one can find $\mu_1$, $\mu_2$ and $y''$ in \eqref{eq_plane_T}, such that $X_j$ is tangential to $\p T^{-1}(\Omega_{y''})$.  
 Choosing first $\mu_1=e_1$, $\mu_2=e_n$,  $\mu_3=e_2,\dots,\mu_n=e_{n-1}$, and $y''=(x_2,\dots,x_{n-1})$ fixed near $(a_2,\dots,a_{n-1})$, we have that  in a neighborhood of $a$, $\p T^{-1}(\Omega_{y''})$ is smooth and is  given by 
 $x_n=\phi(x_1,x_2,\dots,x_{n-1})$. The tangential vector field to $\p T^{-1}(\Omega_{y''})$ is precisely $X_1(x')$, and therefore, $X_1\varphi=0$ near $a$. Proceeding similarly and choosing $\mu_1=e_j$, $j=2,\dots,n-1$, and $\mu_2=e_n$,  we obtain that  the vector fields $X_j$, given in  \eqref{eq_X_j}, are  
  tangential to the boundaries of the corresponding sections of $\Omega$. It follows that $\varphi$ is locally constant along the boundary $\p \Omega$,  and since $\p \Omega$ is connected, we conclude that $\varphi$ is constant along $\p \Omega$. The proof is complete.

\end{proof}

To finish the proof of Theorem \ref{thm_main_3},  we argue in the same way as at the end of Section \ref{sec_thm1}. The proof of Theorem \ref{thm_main_3} is complete.

\begin{appendix}

\section{Characterizing curl--free vector fields} 

Let $\Omega\subset \R^n$, $n\ge 2$, be a bounded domain with $C^1$ boundary and let us denote by $C^{0,1}(\overline{\Omega})$ the space of Lipschitz continuous functions in $\Omega$ and by $C^{1,1}(\overline{\Omega})=\{\varphi\in C^1(\Omega): \nabla \varphi\in C^{0,1}(\overline{\Omega})\}$.

The following result, used in Section 5 in the reconstruction of the vector field part of the perturbation, may be of some independent interest. 
Notice that it holds without any topological assumptions on the domain $\Omega\subset\R^n$. 

\begin{prop}
\label{prop_existence_varphi}
Let $\Omega\subset \R^n$, $n\ge 2$, be a bounded domain with $C^1$ boundary,  $A\in W^{1,\infty}(\Omega,\C^n)$, and 
\begin{equation}
\label{eq_prop_magnetic}
\mu\cdot \int_\Omega e^{ix\cdot \xi} A(x)dx=0, \quad \textrm{for all } \xi,\mu\in \R^n,\ \xi\cdot\mu=0. 
\end{equation}
Then there exists $\varphi\in C^{1,1}(\overline{\Omega})$ such that $A=\nabla\varphi$. 
\end{prop}

 \begin{proof}
 
 We start by following an argument from  \cite{EskRal_1995}.
 Let $\chi_\Omega$ be the characteristic function of $\Omega$. Then \eqref{eq_prop_magnetic} can be written as follows, 
 \begin{equation}
\label{eq_prop_magnetic_1}
 \mu\cdot \hat{\chi_\Omega A }(\xi)=0, \quad \textrm{for all } \xi,\mu\in \R^n,\ \xi\cdot\mu=0.
 \end{equation}
 For any vector $\xi\in\R^n$,  we have the following decomposition,
 \[
 \hat{\chi_\Omega A }(\xi)=\hat{\chi_\Omega A }_\xi(\xi)+ \hat{\chi_\Omega A }_{\perp}(\xi),
 \]
 where $\Re\hat{\chi_\Omega A }_\xi(\xi), \Im \hat{\chi_\Omega A }_\xi(\xi)$ are multiples of $\xi$,  and $\Re\hat{\chi_\Omega A }_{\perp}(\xi), \Im\hat{\chi_\Omega A }_{\perp}(\xi)$ are orthogonal to $\xi$.  It follows from \eqref{eq_prop_magnetic_1} that $\hat{\chi_\Omega A }_{\perp}(\xi)=0$, and therefore, 
 \begin{equation}
\label{eq_prop_magnetic_2}
 \hat{\chi_\Omega A }(\xi)=\alpha(\xi)\xi, \quad \alpha(\xi)= \frac{\hat{\chi_\Omega A }(\xi)\cdot\xi}{|\xi|^2}. 
 \end{equation}
 Moreover, \eqref{eq_prop_magnetic_1} implies that $\hat{\chi_\Omega A }(0)=0$. Since $\chi_\Omega A\in L^1(\R^n)\cap \mathcal{E}'(\overline{\Omega})$, we conclude that $\hat{\chi_\Omega A }(\xi)\in L^\infty(\R^n)\cap C^\infty(\R^n)$. Hence, $\alpha\in L^\infty(\R^n)$.  
 Let 
 \[
\tilde \varphi=\mathcal{F}^{-1}(\alpha(\xi))\in \mathcal{S}'(\R^n)
 \]
 be the inverse Fourier transform of $\alpha$. 
 Then \eqref{eq_prop_magnetic_2} implies that 
$ \chi_\Omega A_j=D_j\tilde \varphi $ in  $\R^n$,  in the sense of distribution theory, and  therefore, 
\[
A_j=D_j\tilde \varphi\quad \textrm{in}\quad  \Omega.
\]
Setting $\varphi=i^{-1}\tilde \varphi$,  we should check that $\varphi\in C^{1,1}(\overline{\Omega})$.   As $D_j\varphi\in W^{1,\infty}(\Omega)= C^{0,1}(\overline{\Omega})$, $j=1,\dots,n$, 
by \cite[Theorem 4.5.12 and Theorem 3.1.7]{Horm_book_1} it follows that $\varphi\in C^{1,1}(\overline{\Omega})$. The proof is complete.

 \end{proof}

\end{appendix}

\section*{Acknowledgements}  

The research of K.K. is supported by the
Academy of Finland (project 125599).  The research of M.L. 
is partially supported 
 by the Academy of Finland Center of Excellence programme 213476. The research of
G.U. is partially supported by the National Science Foundation.

\end{document}